\documentclass[conference]{IEEEtran}
\usepackage{cite}

\ifCLASSINFOpdf
 \usepackage[pdftex]{graphicx}
\else
\fi
\usepackage[cmex10]{amsmath}
\usepackage{array}
\usepackage{mdwmath}
\usepackage{mdwtab}
\usepackage{eqparbox}
\usepackage{url}
\usepackage{latexsym}
\usepackage{amssymb}

\newtheorem{theorem}{Theorem}
\newtheorem{definition}{Definition}
\newtheorem{proposition}{Proposition}
\begin{document}
%
% paper title
% can use linebreaks \\ within to get better formatting as desired

\title{DIALECTICS OF COUNTING AND MEASURES OF ROUGH THEORIES}

% author names and affiliations
% use a multiple column layout for up to three different
% affiliations
\author{\IEEEauthorblockN{A. Mani}
\IEEEauthorblockA{Calcutta Mathematical Society\\
9/1B, Jatin Bagchi Road\\
Kolkata-700029, India\\
Email: a.mani.cms@gmail.com\\
Home Page: \url{http://www.logicamani.co.cc}\\
NCETSC'2011, 2nd and 3rd Feb'2011, Nowrosjee Wadia College and Pune Univ., Pune, India}}

\date{NCETSC'2011, 2nd and 3rd Feb'2011, Nowrosjee Wadia College and Pune Univ., Pune, India}

%\titlerunning{Rough Counting}

\maketitle

%%%%%%%%%%%%%%%%% Restrict to dialectics of counting and add new theorems to measures%%%%%%%%%%%%%%
%%%

\begin{abstract}
New concepts of rough natural number systems, recently introduced by the present author, are used to improve most rough set-theoretical measures in general Rough Set theory (\textsf{RST}) and measures of mutual consistency of multiple models of knowledge. In this research paper, the explicit dependence on the axiomatic theory of granules of \cite{AM99} is reduced and more results on the measures and representation of the numbers are proved.  
\end{abstract}

%\begin{keyword}
%Mathematics of Vagueness \sep Rough Natural Number Systems \sep Axiomatic Theory of Granules \sep Granulation \sep Granular Rough Semantics \sep Algebraic Semantics \sep Rough Y-Systems \sep Cover Based Rough Set Theories \sep Rough Inclusion Functions \sep Measures of Knowledge 
%% MSC codes here, in the form: \MSC code \sep code
%% or \MSC[2008] code \sep code (2000 is the default)
%\MSC 03E70 \sep 03A99 \sep 03E99 \sep 03G25 \sep 94D99
 
%\end{keyword}

\section{INTRODUCTION}

Rough and Fuzzy set theories have been the dominant approaches to vagueness and approximate reasoning from a mathematical perspective. In rough set theory (\textsf{RST}), vague and imprecise information are dealt with through binary relations (for some form of indiscernibility) on a set or covers of a set or through more abstract operators. In classical \textsf{RST}  \cite{ZPB}, starting from an approximation space consisting of a pair of a set and an equivalence over it, approximations of subsets of the set are constructed out of equivalence partitions (these are crisp or definite) that are also regarded as granules in many senses. Most of the developments in \textsf{RST} have been within the ZFC or ZF set-theoretic framework of mathematics. In such frame works, rough sets can be seen as pairs of sets of the form $(A,\,B)$, with $A\subseteq B$ or more generally as in the approaches of the present author as collections of some sense definite elements of the form $\{a_1, a_2, \ldots a_{n},\,b_{1}, b_{2},\ldots b_r \}$ subject to $a_{i}$s being 'part of' of some of the $b_{j}$s (in a Rough $Y$-system) \cite{AM99}.  

Relative \textsf{RST}, fuzzy set theory may be regarded as a complementary approach or as a special case of \textsf{RST} from the membership function perspective (\cite{ZP5}). Hybrid rough-fuzzy and fuzzy-rough variants have also been studied. In these a partitioning of the meta levels can be associated for the different types of phenomena, though it can be argued that these are essentially of a rough set theoretic nature. All of these approaches have been confined to ZFC or ZF or the setting of classical mathematics. Exceptions to this trend include the Lesniewski-mereology based approach \cite{PS3}. Though Rough $Y$-systems have been introduced by the present author in ZF compatible settings \cite{AM99}, they can be generalised to semi-sets and variants in a natural way. This semi-set theoretical variant is work in progress. 

Granules are essentially the building blocks of approximations in \textsf{RST}, but are more usually viewed from contextual perspectives in the literature. As classical \textsf{RST} is generalised to more general relations and covers, the process of construction and definition of approximations becomes more open ended and draws in different amounts of arbitrariness or hidden principles. The relevant concept of 'granules', the things that generate approximations, may also become opaque in these settings. To address these issues and for semantic compulsions a new axiomatic theory of granules has been developed over a \textsf{RYS} in \cite{AM99} and other recent papers by the present author. This theory has been used to formalize different principles including the local clear discernibility principle in the same paper. In a separate paper, it is extended to various types of general \textsf{RST} and is used to define the concepts of discernibility used in counting procedures, generalised measures and the problems of representation of semantics.  

Many types of contexts involving vagueness cannot be handled in a elegant way through standard mathematical techniques. Much of \textsf{RST} and \textsf{FST} are not known to be particularly elegant in handling different measures like the degree of membership or inclusion. For example these measures do not determine the semantics in clear terms or the semantics do not determine the measures in a unique way. But given assumptions about the measures, compatible semantics (\cite{PL}) in different forms are well known. This situation is due to the absence of methods of counting collections of objects including relatively indiscernible ones and methods for performing arithmetical operations on them. However in various algebraic semantics of classical \textsf{RST} some boundaries on possible operations may be observed.

The process of counting a set of objects given the restriction that some of these may be indiscernible within themselves may appear to be a very contextual matter and on deeper analysis may appear to bear no easy relationship with any fine structure concerning the vagueness of collection of such elements or the rough semantics (algebraic or frame). This is reflected in the absence of related developments on the relationship between the two in the literature and is more due to the lack of developments in the former. 

In a recent research paper, theories of \emph{vague numbers} or rather procedures for counting collections of objects including indiscernible ones have been introduced by the present author and have been applied to extend various measures of \textsf{RST}. The extended measures have better information content and also supplement the mereological theory from a quantitative perspective. Proper extension of these to form a basis of \textsf{RST} within ZF/ZFC is also developed in the paper. Here, by a 'basis of \textsf{RST}', I mean a theory analogous to the theory of numbers from which all mathematics of exact phenomena can be represented. All of this is based over an extended version of my axiomatic theory of granules for \textsf{RST} \cite{AM99}.

In the present paper, the counting procedures and generalised measures are presented in a more accessible way without explicit dependence on the the axiomatic granularity theory and independent of the programme of \emph{mathematics of vagueness}. This is meaningful because the improvements in the quality of measures is substantial and relevant in all approaches. More results are also proved on the representation of these new rough naturals.

In the second section, the basic orientation of object and meta levels and some non-standard (with respect to the literature on \textsf{RST}) examples are initially presented. In the third subsection, aspects of counting in domains of vague reasoning are explained. Dialectical counting processes are introduced in the third section. These are used to generalise rough inclusion functions, degrees of knowledge dependency and other measures in the next section. Subsequently I examine the problem of improving the representation of counts in some detail in the fifth section.    

\subsection{SOME BACKGROUND}

Some background is provided for convenience. In classical \textsf{RST}, an approximation space (\textsf{AS}) is a pair $S=\left\langle \underline{S},\,R\right\rangle$, with $\underline{S}$ being a set and $R$ an equivalence relation over it. The lower and upper approximation of an element $A$ of the power set $\wp (S)$ are defined by $A^{l}=\bigcup\{[x]:\,[x]\,\subseteq\, A\}$ and $A^{u}=\bigcup \{[x]:\,[x]\cap\,A\neq\,\emptyset\}$ with $[x]$ denoting the equivalence class generated by an element $x \in S$. Here the equivalence classes function as granules. If $R$ is replaced by a tolerance relation $T$, then we get a tolerance approximation space (\textsf{TAS}). Some references for extension of classical RST to \textsf{TAS} are \cite{KM}, \cite{CC} \cite{SW,SW3} and \cite{AM105}. Partial equivalences are used instead of equivalences in esoteric \textsf{RST} \cite{AM24}.

The theory of \emph{Rough Orders} \cite{IT2} is a generalisation of rough set theory. Let $\langle\mathcal{B},\xi\rangle$ be a pair with $\mathcal{B} = \langle\underline{B},\leq\rangle$ being a bounded poset and $\xi = \langle\underline{E},\leq'\rangle$ a bounded lattice with $\leq'
={\leq\cap\underline{E}^2}$. $1_{\mathcal{B}}={1_\xi}$,$0_{\mathcal{B}}={0_xi}$,then
the pair is a \emph{rough order}. In a poset $\mathcal{B}$, a self-map $u$ is called a \emph{pre-closure} or \emph{upper approximation map} if and only if $u^2 = u$ and $\forall{x},x\leq{u(x)}$. A pre-closure map $v$ on the dual $\mathcal{B}^{d} = \langle\underline{B},\leq'\rangle$ is called a \emph{lower approximation map}. Some of the basic results of rough orders include the if and only if condition for approximation maps to be closure maps and the notions of degree of roughness. If $\underline{\xi}$ and $\overline{\xi}$ are the (\cite{IT2}) lower and upper approximation operators on $\langle\mathcal{B},\xi\rangle$ then any element $x$ satisfying $\underline{\xi}(x) = {x} = \overline{\xi}(x)$ is an \emph{exact element} and conversely. 

Cover based \textsf{RST} can be traced to \cite{WZ}, where starting from the cover $\{[x]_{T}; x\in S\}$ the approximations $A^{l}$ and $A^{u}$  are defined. A 1-neighbourhood \cite{YY9} $n(x)$ of an element $x\in S$ is simply a subset of $S$. The collection of all 1-neighbourhoods $\mathcal{N}$ of $S$ will form a cover if and only if $(\forall x)(\exists y) x\in n(y)$ (anti-seriality). So in particular a reflexive relation on $S$ is sufficient to generate a cover on it. Of course, the converse association does not necessarily happen in a unique way. 1-neighbourhoods can be used to investigate the relation between many types of approximations that are definable using 1-neighbourhoods. But there are examples of approximations that do not fit into this general framework (\cite{AM99}).     

General Rough Y-Systems (\textsf{RYS+}) and Rough Y-Systems (\textsf{RYS}) were introduced in \cite{AM99}. 
These are intended to capture a minimal common fragment of most RSTs. The former is more efficient due to its generality. Both \textsf{RYS+} and \textsf{RYS} can be seen as the generalization of the algebra formed on the power set of the approximation space in classical \textsf{RST} and rough orders. $\mathbf{P} xy$ can be read as 'x is a part of y' and is intended to generalise inclusion in the classical case. The elements in $S$ may be seen as the collection of approximable and exact objects - this interpretation is compatible with $S$ being a set. The description operator of FOPL $\iota$ is used in the sense: $\iota(x) \Phi(x)$ means 'the $x$ such that $\Phi(x)$'. It helps in improving expression and given this the meta-logical $','$ can be understood as $\wedge$ (which is part of the language). The description operator actually extends FOPL by including more of the metalanguage and from the meaning point of view is different, though most logicians will see nothing different. For details, the reader may refer to \cite{WH}.

\begin{definition}
A \emph{Rough Y System} (\textsf{RYS+}) will be a tuple of the form  \[\left\langle {S},\,W,\,\mathbf{P} ,\,(l_{i})_{1}^{n},\,(u_{i})_{1}^{n},\,+,\,\cdot,\,\sim,\,1 \right\rangle \] satisfying all of the following ($\mathbf{P}$ is intended as a binary relation on $S$ and $W\,\subset\,S$, $n$ being a finite positive integer. $\iota$ is the description operator of FOPL: $\iota(x) \Phi(x)$ means 'the $x$ such that $\Phi(x)$ '. ',' can be read as a $\wedge$ or as 'and' in the meta language):
\begin{enumerate}
\item {$(\forall x) \mathbf{P} xx$ ; $(\forall x,\,y)(\mathbf{P} xy,\,\mathbf{P} yx\,\longrightarrow\,x=y)$}
\item {For each $i,\,j$, $l_{i}$, $u_{j}$ are surjective functions $:S\,\longmapsto\,W$ }
\item {For each $i$, $(\forall{x,\,y})(\mathbf{P} xy\,\longrightarrow\,\mathbf{P} (l_{i}x)(l_{i}y)\, \wedge\, \mathbf{P} (u_{i}x)(u_{i}y))$}
\item {For each $i$, $(\forall{x})\,\mathbf{P} (l_{i}x)x \,\wedge\,\mathbf{P} (x)(u_{i}x))$}
\item {For each $i$, $(\forall{x})(\mathbf{P} (u_{i}x)(l_{i}x)\,\longrightarrow\,x=l_{i}x=u_{i}x)$}
\end{enumerate}

The operations $+,\,\cdot$ and the derived operations $\mathbf{O},\, \mathbb{P},\,\mathbf{U},\,\mathbb{X},\,\mathbb{O} $ will be assumed to be defined uniquely as follows:

\begin{itemize}
\item{Overlap: $\mathbf{O} xy\,\mathrm{iff} \,(\exists z)\,\mathbf{P} zx\,\wedge\,\mathbf{P} zy$}
\item{Underlap: $\mathbf{U} xy\,\mathrm{iff} \,(\exists z)\,\mathbf{P} xz\,\wedge\,\mathbf{P} yz$}
\item{Proper Part: $\mathbb{P} xy\, \mathrm{iff}\, \mathbf{P} xy\wedge\neg \mathbf{P} yx$}
\item{Overcross: $\mathbb{X} xy \,\mathrm{iff}\,\mathbf{O} xy\wedge\neg \mathbf{P} xy$}
\item{Proper Overlap: $\mathbb{O} xy \,\mathrm{iff}\,\mathbb{X} xy \,\wedge\,\mathbb{X} yx $}
\item{Sum: $x+y=\iota z (\forall w)(\mathbf{O} wz\, \leftrightarrow\,(\mathbf{O} wx \vee \mathbf{O} wy))$ }
\item{Product: $x \cdot y=\iota z (\forall w)(\mathbf{P} wz\,\leftrightarrow\,(\mathbf{P} wx \wedge \mathbf{P} wy))$ }
\item{Difference: $x - y=\iota z (\forall w)(\mathbf{P} wz\,\leftrightarrow\,(\mathbf{P} wx \wedge \neg \mathbf{O} wy))$ }
\item{Associativity: We will assume that $+,\,\cdot$ are associative operations and so the corresponding operations on finite sets will be denoted by $\oplus,\,\odot$ respectively.}
\end{itemize}
\end{definition}

\begin{flushleft}
\textbf{Remark:} 
\end{flushleft}
$W$ can be dropped from the above definition and it can be required that the range of the operations $u_{i}, l_{j}$ are all equal for all values of $i, j$.   

\begin{definition}
In the above definition, if we relax the surjectivity of $l_i ,u_i $, require partiality of the operations $+$ and $\cdot$, weak associativity instead of associativity and weak commutativity instead of commutativity, then the structure will be called a \emph{General Rough Y-System} (\textsf{RYS}). In formal terms, we mean
\begin{itemize}
\item{Sum1: $x + y=\iota z (\forall w)(\mathbf{O} wz\, \leftrightarrow\,(\mathbf{O} wx \vee \mathbf{O} wy))$ if defined}
\item{Product1: $x \cdot y=\iota z (\forall w)(\mathbf{P} wz\,\leftrightarrow\,(\mathbf{P} wx \wedge \mathbf{P} wy))$ if defined}
\item{Weak Associativity: $x\oplus (y\oplus z)\,\stackrel{\omega *}{=}(x\oplus y)\oplus z$ \\ and similarly for product. The weak equality $A \stackrel{\omega *}{=} B$ for two terms $A, B$, essentially means if either side is defined, then the other is and the two terms are equal.}
\item{Weak Commutativity: $x\oplus y\,\stackrel{\omega *}{=}\,y\oplus x $;  $x\cdot y\,\stackrel{\omega *}{=}\,y\cdot x$}
\end{itemize}
\end{definition}

For more details on \textsf{RYS}, the reader is referred to \cite{AM99}. \textsf{RYS} can be naturally used to model collections of vague objects, exact objects and others. In the present paper, the focus will be on such collections and definable rough equalities. The reference to other structure will be implicit and it is perfectly possible to use the developed theory in the context of the other mentioned theories. 

In almost all applications, the collection of all granules $\mathcal{G}$ forms a subset of the \textsf{RYS} $S$. But a more general setting can be of some interest especially in a semi-set theoretical setting. This aspect will be considered separately.

\section{SOME EXPLANATION}

In this section, some of the basic motivations, examples and aspects of the context are explained. One advantage of \textsf{RYS} is that it becomes possible to deal with sets of approximations of objects, vague objects and other objects with no knowledge of their ontology. In other words the evolution of the scenario need not be fully known beforehand. The same applies to contexts in which the general counting procedures would be applicable. This necessitates some clarifications of the basic relationship between object and meta levels.   

In classical RST (see \cite{ZPB}), an approximation space is a pair of the form $\left\langle S,\,R \right\rangle $, with $R$ being an equivalence on the set $S$. On the power set $\wp (S)$ lower and upper approximation operators, apart from the usual Boolean operations, are definable. The resulting structure constitutes a semantics for RST (though not satisfactory) in a classicalist perspective. This continues to be true even when $R$ is some other type of binary relation. More generally (see fourth section) it is possible to replace $\wp (S)$ by some set with a parthood relation and some approximation operators defined on it. The associated semantic domain in the sense of a collection of restrictions on possible objects, predicates, constants, functions and low level operations on those will be referred to as the classical semantic domain for general RST. In contrast the semantics associated with sets of roughly equivalent or relatively indiscernible objects with respect to this domain will be called the rough semantic domain. Actually many other semantic domains including hybrid semantic domains can be generated (see \cite{AM699,AM105,AM3,AM2}) for different types of rough semantics, but these two broad domains will always be. In one of the semantics developed in \cite{AM105}, the reasoning is within the power set of the set of possible order-compatible partitions of the set of roughly equivalent elements.  The concept of \emph{semantic domain} used is therefore similar to the sense in which it is used in general abstract model theory \cite{MD}. 

Formal versions of these types of semantic domains will be useful for clarifying the relation with categorical approaches to fragments of granular computing \cite{BY}. But even without a formal definition it can be seen that the two approaches are not equivalent. Since the categorical approach requires complete description of fixed type of granulations, it is difficult to apply and especially when granules evolve relative particular semantics or semantic domains. The entire category \textbf{ROUGH} of rough sets in \cite{BY}, assumes a uniform semantic domain as evidenced by the notion of objects and morphisms used therein. A unifying semantic domain may not also be definable for many sets of semantic domains in our approach. This means the categorical approach needs to be extended to provide a possibly comparable setting.

The term \emph{object level} will mean a description that can be used to directly interface with fragments (sufficient for the theories or observations under consideration) of the concrete real world. Meta levels concern fragments of theories that address aspects of dynamics at lower meta levels or the object level. Importantly, we permit meta level aspects to filter down to object levels relative a different object level specification. So it is always possible to construct more meta levels and expressions constructed in these can be said to carry intentions. 

\subsection{EXAMPLES}

The examples in this section do not explicitly use information or decision tables though all of the information or decision table examples used in various \textsf{RSTs} would be naturally relevant for the general theory developed. Here I will be laying emphasis on demonstrating the need for many approximation operators, choice among granules, different concepts of rough equalities and conflict resolution.    

\begin{flushleft}
\textbf{Example-1:}\\ 
\end{flushleft}

Consider the following statements that can be associated with the description of an apple kept in a plate on a table:

\begin{itemize}
\item {Object is apple-shaped; Object has maroon colour}
\item {Object has vitreous skin; Object has shiny skin}
\item {Object has vitreous, smooth and shiny skin}
\item {Green apples are not sweet to taste}
\item {Object does not have coarse skin as some apples do}
\item {Apple is of variety A; Apple is of variety X}
\end{itemize}
Some of the individual statements like those about shape, colour and nature of skin may be 'atomic' in the sense that a more subtle characterization may not be available. It is also obvious that only some subsets of these statements may consistently apply to the apple on the table. This leads to the question of selecting some atomic statements over others. But this operation is determined by consistency of all the choices made. Therefore from a \textsf{RST} perspective, the atomic statements may be seen as granules and then it would also seem that choice among sets of granules is natural. More generally 'consistency' may be replaced by more involved criteria that are determined by the goals. A nice way to see this would be to look at the problem of discerning the object in different contexts - discernibility of apples on trees require different kind of subsets of granules.

\begin{flushleft}
\textbf{Example-2:}\\ 
\end{flushleft}

In the literature on educational research \cite{PD} it is known that even pre-school going children have access to powerful mathematical ideas in a form. A clear description of such ideas is however not known and researchers tend to approximate them through subjective considerations. For example, consider the following example from \cite{PD}:

Four-year-old Jessica is standing at the bottom of a small rise in the preschool yard when she is asked by another four-year-old on the top of the rise to come up to her.
\begin{itemize}
 \item {``No, you climb down here. It’s much shorter for you.``}
\end{itemize}

The authors claim that ''Jessica has adopted a developing concept of comparison of length to solve —
at least for her the physical dilemma of having to walk up the rise''. But a number of other concepts like 'awareness of the effects of gravitational field', 'climbing up is hard', 'climbing up is harder than climbing down', 'climbing down is easier', 'climbing up is harder', 'others will find distances shorter', 'make others do the hard work' may or may not be supplemented by linguistic hedges like \emph{developing} or \emph{developed} and assigned to Jessica. The well known concept of concept maps cannot be used to visualize these considerations, because the concept under consideration is not well defined. Of these concepts some will be assuming more and some less than the actual concept used and some will be closer than others to the actual concept used. Some of the proposals may be conflicting, and that can be a problem with most approaches of \textsf{RST} and fuzzy variants. The question of one concept being closer than another may also be dependent on external observers. For example, how do 'climbing up is harder' and 'climbing up is harder than climbing down' compare?

The point is that it makes sense to:

\begin{itemize}
\item {accommodate multiple concepts of approximation and rough equalities}
\item {assume that subsets of granules may be associated with each of these approximations }
\item {assume that disputes on 'admissible approximations' can be resolved by admitting more approximations}
\end{itemize}

It is these considerations and the actual reality of different \textsf{RST} that are among the motivations for the definition of Rough $Y$-systems.   

\subsection{OBJECTIVITY OF MEASURES AND GENERAL RST}

In \textsf{RST} different types of rough membership and inclusion functions are defined using cardinality of associated sets. When fuzzy aspects are involved then these types of functions become more controversial as they tend to depend on the judgement of the user. Often these are due to attempts at quantifying linguistic hedges. Further these types of functions are also dependent on the way in which the evolution of semantics is viewed. But even without regard to this evolution, the computation of rough inclusion functions invariably requires one to look at things from a higher meta level - from where all objects appear exact. In other words an estimate of a perfect granular world is also necessary for such computations and reasoning.

Eventually this leads to mix up of information obtained from perceiving things at different levels of granularity. I find all this objectionable from the point of view of rigour and specific applications too. To be fair such a mix up seems to work fine without much problems in many applications. But that is due to the scope of the applications and the fact that oversimplifications through the use of cardinality permits a second level of 'intuitive approximation'. 

\subsection{NUMBERS AND THEIR GENERALIZATION}

In the mathematics of exact phenomena, the natural numbers arise in the way they do precisely because it is assumed that things being counted are well defined and have exact existence. When a concrete collection of identical balls on a table are being counted, then it is their relative position on the table that helps in the process. But there are situations in real life, where 
\begin{itemize}
\item {such identification may not be feasible}
\item {the number assigned to one may be forgotten while counting subsequent objects}
\item {the concept of identification by way of attributes may not be stable}
\item {the entire process of counting may be 'lazy' in some sense.}
\item {the mereology necessary for counting may be insufficient.}
\end{itemize}

Some specific examples of such contexts are:
\begin{itemize}
\item {Direct counting of fishes in a lake is difficult and the sampling techniques used to estimate the population of fishes do not go beyond counting samples and pre-sampling procedures.                                                                            For example some fishes may be caught, marked and then put back into the lake. Random samples may be drawn from the mixed population to estimate the whole population using proportion related statistics. The whole procedure however does not involve any actual counting of the population.}
\item {In crowd management procedures, it is not necessary to have exact information about the actual counts of the people in crowds.}
\item {In many counting procedures, the outcome of the last count (that is the total number of things) may alone be of interest. This way of counting is known to be sufficient for many apparently exact mathematical applications.}
\item {Suppose large baskets containing many varieties of a fruit are mixed together and suppose an observer with less-than-sufficient skills in classifying the fruits tries to count the number of fruits of a variety.  The problem of the observer can be interpreted in mereological terms.}
\item {Partial algebras are very natural in a wide variety of mathematics. For example, in semigroup theory the set of idempotents can be endowed with a powerful partial algebraic structure. Many partial operations may be seen to originate from failure of standard counting procedures in specific contexts.}
\end{itemize}

Though not usually presented in the form, studies of group actions, finite and infinite permutation groups and related automorphisms and endomorphisms can through light on lower level counting. In the mathematics of exact phenomena, these aspects would seem superfluous because cardinality is not dependent on the order of counting. 
But in the context of the present generalization somewhat related procedures are seen to be usable for improving representation in the penultimate section.  

In counting collections of objects including relatively exact and indiscernible objects, the situation is far more complex - the first thing to be modified would be the relative orientation of the object and different meta levels as counting in any sense would be from a higher meta level. Subsequently the concept of counting (as something to be realized through injective maps into $N$) can be modified in a suitable way. The eventual goal of such procedures should be the attainment of order-independent representations. 

\subsection{ROUGH EQUALITIES}

Rough equalities of various types have been studied in the literature. Typically they depend on one or more approximation operators for their definition. They can almost always be seen to arise as Boolean combination of atomic equalities of the the form \emph{if the approximation of two things are equal then they are equivalent in a rough sense}. In classical \textsf{RST}, two subsets $A, B$ of an approximation space $S$ are \emph{roughly equal} if and only if $A^{l} = B^{l}$ and $A^{u}= B^{u}$, that is if the two sets are bottom- equal and top-equal \cite{ZPB}.  In the theory of rough orders, rough equalities can be defined in an analogous way. In other \textsf{RST}, the number of useful rough equalities increase in number. For example, in bitten rough set theory \cite{AM105}, as many as six rough equalities are of interest. In \textsf{RYS}, many more atomic equalities of interest can also be defined. 

In the two examples above, the concept of discernibility of objects or concepts can be naturally made to depend on the selected attributes and this in turn induces related indiscernibility. Also note that the relation may be a tolerance that is not transitive.

\section{DIALECTICAL COUNTING, MEASURES}

To count a collection of objects in the usual sense it is necessary that they be distinct and well defined. So a collection of well defined distinct objects and indiscernible objects can be counted in the usual sense from a higher meta level of perception. Relative this meta level, the collection must appear as a set. In the counting procedures developed, the use of this meta level is minimal and certainly far lesser than in other approaches.     
It is dialectical as two different interpretations are used simultaneously to complete the process. These two interpretations cannot be merged as they have contradictory information about relative discernibility. Though the classical interpretation is at a relatively higher meta level, it is still being used in the counting process and the formulated counting is not completely independent of the classical interpretation. A strictly formal approach to these aspects will be part of a forthcoming paper.    

Counting of a set of objects of an approximation space and that of its power set are very different as they have very different kinds of indiscernibility inherent in them. The latter possess a complete evolution for all of the indiscernibility present in it while the former does not. Counting of elements of a \textsf{RYS} is essentially a generalisation of the latter. In general any lower level structure like an approximation space corresponding to a 1-neighbourhood system (\cite{YY9}) or a cover need not exist in any unique sense. The axiomatic theory of granules developed in \cite{AM99} provides a formal exactification of these aspects.  

Let $S$ be a \textsf{RYS}, with $R$ being a derived binary relation (interpreted as a weak indiscernibility relation) on it. As the other structure on $S$ will not be explicitly used in the following, it suffices to take $S$ to be an at most countable set of elements in ZF, so that it can be written as a sequence of the form: $\{x_{1},\,x_{2},\,\ldots,\,x_{k},\,\ldots ,\,\}$. Taking $(a,b) \in R$  to mean '$a$ is weakly indiscernible from $b$' concepts of \emph{primitive counting} regulated by different types of meta level assumptions are defined below. The adjective \emph{primitive} is intended to express the minimal use of granularity and related axioms.   

\subsection*{Indiscernible Predecessor Based Primitive Counting (IPC)}

In this form of 'counting', the relation with the immediate preceding step of counting matters crucially.
\begin{enumerate}
\item {Assign $f(x_{1})\, = \,1_{1} =s^{0}(1_{1})$}
\item {If $f(x_{i})=s^{r}(1_{j})$ and $(x_{i},x_{i+1}) \in  R$, then assign $f(x_{i+1})=1_{j+1}$ }
\item {If $f(x_{i})=s^{r}(1_{j})$ and $(x_{i},x_{i+1}) \notin  R$, then assign $f(x_{i+1})=s^{r+1}(1_{j})$ }
\end{enumerate}

The 2-type of the expression $s^{r+1}(1_{j})$ will be $j$. Successors will be denoted by the natural numbers indexed by 2-types. An alternative obvious way of writing IPC counts as sequences would be \[\{(\alpha(i),\,\beta(i))\}_{i},\]
where the index $i$ ranges over $N$, $\beta(i)$ would be the 2-type and $\alpha(i)$ is a 0 or a successor.

\begin{proposition}
Sequences of IPC of the form $\{(\alpha(i),\,\beta(i))\}_{i}$ satisfy all of the following (when the quantities $\alpha(i), \beta(i)$s are interpreted as natural numbers):
\begin{itemize}
\item {$(i < j, \beta(i)=\beta(j)\longrightarrow \alpha(i) = \alpha(j) + j -i ) $}
\item {$(i < j < k, \beta(i) = \beta(k)\longrightarrow \beta(i) = \beta(j) ) $} 
\end{itemize}
\end{proposition}

\subsection*{History Based Primitive Counting (HPC)}
In HPC, the relation with all preceding steps of counting will be taken into account. 
\begin{enumerate}
\item {Assign $f(x_{1}) =  1_{1} = s^{0}(1_{1})$}
\item {If $f(x_{i}) = s^{r}(1_{j})$ and $(x_{i},x_{i+1}) \in  R$, then assign $f(x_{i+1}) = 1_{j+1}$ }
\item {If $f(x_{i}) = s^{r}(1_{j})$ and $\bigwedge_{k < i+1} (x_{k},x_{i+1}) \notin R$, then assign $f(x_{i+1}) = s^{r+1}(1_{j})$ }
\end{enumerate}

In any form of counting in this section, if $f(x)=\alpha$, then $\tau (\alpha)$ will denote the least element that is related to $x$, while $\epsilon (\alpha)$ will be the greatest element preceding $\alpha$, which is related to $\alpha$.

\subsection*{History Based Perceptive Partial Counting (HPPC)}
In HPPC, the valuation set shall be $N\,\cup\,\{*\}=N^{*}$ with $*$ being an abbreviation for 'undefined'.
\begin{enumerate}
\item {Assign $f(x_{1}) = 1 = s^{0}(1)$}
\item {If $(x_{i},x_{i+1}) \in  R$, then assign $f(x_{i+1}) = *$ }
\item {If $Max_{k< i} f(x_{k}) = s^{r}(1)$ and $\bigwedge_{k < i} (x_{k},x_{i}) \notin  R$, then assign  $f(x_{i}) = s^{r+1}(1)$ }
\end{enumerate}
Clearly, HPPC depends on the order in which the elements are counted.

\subsection*{Indiscernible Predecessor Based Partial Counting (IPPC)}
This form of counting is similar to HPPC, but differs in using the IPC methodology.
\begin{enumerate}
\item {Assign $f(x_{1}) = 1 = s^{0}(1)$}
\item {If $(x_{i},x_{i+1}) \in  R$, then assign $f(x_{i+1}) = *$ }
\item {If $Max_{k< i} f(x_{k})=s^{r}(1)$ and $(x_{i-1},x_{i})\notin R$, then assign $f(x_{i})=s^{r+1}(1)$ }
\end{enumerate}

\begin{definition}
A generalised approximation space $\left\langle S,\,R\right\rangle $ will be said to be \emph{IPP Countable} (resp HPP Countable) if and only if there exists a total order on $S$, under which all the elements can be assigned a unique natural number under the IPPC (resp HPPC) procedure . The ratio of number of such orders to the total number of total orders will be said to be the \emph{IPPC Index} (resp HPPC Index).
\end{definition}

These types of countability are related to measures, applicability of variable precision rough methods and combinatorial properties. The counting procedures as such are strongly influenced by the meaning of associated contexts and are not easy to compare. In classical \textsf{RST}, HPC can be used to effectively formulate a semantics, while the IPC method yields a weaker version of the semantics in general. This is virtually proved below:

\begin{theorem}
In the HPC procedure applied to classical \textsf{RST}, a set of the form $\{l_{p},\,1_{q},\,\ldots 1_{t}\}$ is a granule if and only if it is the greatest set of the form with every element being related to every other and for any element   $\alpha$ in it $\tau(\alpha)=f^{\dashv}(l_{p})$ and $\epsilon(\alpha)=f^{\dashv}(1_{t})$ (for at least two element sets). Singleton granules should be of the form $\{l_{p}\}$ with this element being unrelated to all other elements in $S$.  
\end{theorem}

\begin{theorem}
All of the following are decidable in a finite number of steps in a finite \textsf{AS} by the application of the HPC counting method :
\begin{enumerate}
\item {Whether a given subset $A$ is a granule or otherwise.}
\item {Whether a given subset $A$ is the lower approximation of a set $B$}
\item {Whether a given subset $A$ is the upper approximation of a set $B$.}
\end{enumerate}
\end{theorem}

\begin{proof}
Granules in the classical \textsf{RST} context are equivalence classes and elements of such classes relate to other elements within the class alone.
\end{proof}

\begin{theorem}
The above two theorems do not hold in the IPC context. The criteria of the first theorem defines a new type of granules in totally ordered approximation spaces and relative this the second theorem may not hold. 
\end{theorem}

\begin{proof}
By the IPC way of counting, equivalence classes can be split up into many parts and the numeric value may be quite deceptive. Looking at the variation of the 2-types, the criteria of the first theorem above will define parts of granules, that will yield lower and upper approximations distinct from the classical ones respectively.    
\end{proof}

Extension of these counting processes to TAS for the abstract extraction of information about granules, becomes too involved even for the simplest kind of granules. However if the IPC way of counting is combined with specific parthood orders, then identification of granules can be quite easy. The other way is to consider entire collections of countings according to the IPC scheme. 

\begin{theorem}
Among the collection of all IPC counts of a RYS $S$, any one with maximum number of $1_{i}$s occurring in succession determines all granules. If $S$ is the power set of an approximation space, then all the definite elements can also be identified from the distribution of sequences of successive $1_{i}$s.   
\end{theorem}

\begin{proof}
\begin{itemize}
\item {Form the collection $\mathcal{I}(S)$ of all IPC counts of $S$}
\item {For $\alpha, \beta\in \mathcal{I}(S)$, let $\alpha \preceq \beta$ if and only if $\beta$ has longer strings of $1_{i}$ s towards the left and more number of strings of $1_{i}$s of length greater than 1, than $\alpha$ }
\item {The maximal elements in this order suffice for this theorem.}
\end{itemize}
\end{proof}

\subsection*{EXAMPLE}

Let $S=\{f, b, c, a, k, i, n, h, e, l, g, m \}$ and let $R,\,Q$ be the reflexive and transitive closure of the relation \[\{(a,b),\,(b, c),\,(e, f),\, (i, k), (l, m),\, (m, n),\,(g, h) \}\] and \[\{(a,b),\,(e,f),\,(i,k),\,(l,m),\,(m,n)\}\] respectively. Then $\left\langle S,\,R\right\rangle $ and $\left\langle S,\,Q\right\rangle $ are approximation spaces. This set can be counted (relative $R$) in the presented order as follows:

\begin{description}
\item [IPC]{$\{1_{1}, 2_{1}, 1_{2},\,1_{3},\, 2_{3}, 1_{4}, 2_{4}, 3_{4}, 1_{5}, 2_{5}, 1_{6}, 2_{6} \}$} 
\item [HPC]{$\{1_{1}, 2_{1}, 1_{2}, 1_{3},\, 2_{3}, 1_{4}, 1_{5}, 2_{5}, 1_{6},\, 1_{7}, 1_{8}, 1_{9} \}$}
\item [HPPC]{$\{1_{1}, 2_{1}, *, *, 3_{1}, *, 4_{1}, 5_{1}, *, *, *, * \}$}
\item [HPC]{Relative $Q$: $\{1_{1}, 2_{1}, 3_{1}, 1_{2}, 2_{2}, 1_{3}, 2_{3}, 3_{3},$\\ $1_{4},1_{5}, 2_{5}, 1_{6}\} $}
\end{description}

The quotient \[S|Q=\{\{a, b\}, \{c\},\{e,f\}, \{i, k\}, \{l, m, n \},\{g\},\{h\}\}\] and the positive region of $Q$ relative $R$ is \[POS_{R}(Q)=\bigcup_{X\in S|Q} X^{l}=\{e,f, l, m, n\} = \{f, n, e, l, m\}\] (in the order). The induced HPC counts of this set are respectively \[\{1_{1}, 2_{5}, 1_{6}, 1_{7}, 1_{9}\}\] and \[\{1_{1}, 2_{3}, 1_{4}, 1_{5}, 1_{6}\}.\]

\section{GENERALIZED MEASURES}

According to Pawlak's approach \cite{ZPB} to theories of knowledge from a classical rough perspective, if $\underline{S}$ is a set of attributes and $R$ an indiscernibility relation on it, then sets of the form $A^{l}$ and $A^{u}$ represent clear and definite concepts. If $Q$ is another stronger equivalence ($Q\,\subseteq\,R$) on $\underline{S}$, then the state of the knowledge encoded by $\left\langle \underline{S},\,Q \right\rangle $ is a \emph{refinement} of that of $S=\left\langle\underline{S},\,R\right\rangle $. The $R$-positive region of Q is defined to be \[POS_{R}(Q)\,=\,\bigcup_{X\in S|Q} X^{l_{R}}\,;\;\;\,X^{l_{R}}\,=\,\bigcup\{[y]_{R};\,[y]_{R}\subseteq X\}.\]The degree of dependence of knowledge $Q$ on $R$ $\delta(Q,R$ is defined by \[\delta(Q, R) = \dfrac{Card(POS_{R}(Q))}{Card (S)}.\]

\begin{definition}
The \emph{granular degree of dependence of knowledge $Q$ on $R$}, $gk(Q, R)$ will be the tuple $(k_{1},\,\ldots k_{r} )$, with  $k_{i}$s being the ratio of the number of granules of type $i$ included in $POS_{R}(Q)$ to $card (S)$. 
\end{definition}

Note that the order on $S$, induces a natural order on the granules (classes) generated by $R,\,Q$ respectively. This vector cannot be extracted from a single HPC count in general (the example in the last section should be suggestive). However if the granulation is taken into account, then much more information (apart from the measure) can be extracted. This aspect is considered in more detail in a forthcoming paper. 

\begin{proposition}
 If $gk(Q, R) = (k_{1}, k_{2},\ldots, k_{r} )$, then $\delta(Q, R) = \sum k_{i}$.
\end{proposition}

The concepts of consistency degrees of multiple models of knowledge introduced in \cite{CHP} can also be improved by a similar strategy:

If $\delta(Q,R) = a$ and $\delta(R,Q) = b$, then the consistency degree of Q and R, $Cons(Q,R)$ is defined by  \[Cons(Q,R) = \dfrac{a+b+nab}{n+2}, \] where $n$ is the consistency constant. With larger values of $n$, the value of the consistency degree becomes smaller.

\begin{definition}
If $gk(Q,R) = (k_{1}, k_{2},\ldots, k_{r} )$ and $gk(R, Q) = (l_{1}, l_{2},\ldots, l_{p} )$ then
the \emph{granular consistency degree} $gCons(Q, R)$ of the knowledge represented by $Q, R$, will be   
\[gCons(Q,R) = ({k_1^*},\ldots, {k_r^*}, {l_1^*},\ldots, {l_p^*},{n k_{1}^{*}  l_{1}}, \dots {n k_{r}^{*} l_{p}}),\] where $k_i^* = \frac{k_{i}}{n+2}$ for $i = 1, \ldots, r$ and $l_j^*=\frac{l_j}{n+2}$ for $j=1, \ldots, p$.
\end{definition}

The knowledge interpretation can be extended in a natural way to other general RST (including \textsf{TAS}) and also to choice inclusive rough semantics \cite{AM99}. Construction of similar measures is however work in progress. With respect to the counting procedures defined, these two general measures are relatively constructive provided granules can be extracted. This is possible in many of the cases and not always. They can be replaced by a technique of defining atomic sub-measures based on counts and subsequently combining them. 
These aspects will be taken up in a future paper.

\subsection{ROUGH INCLUSION FUNCTIONS}

Various rough inclusion and membership functions with related concepts of degrees are known in the literature (see \cite{AG3} and references therein). If these are not representable in terms of granules through term operations formed from the basic ones (\cite{AM99}), then they are not truly functions/degrees of the rough domain. To emphasize this aspect, I will refer to such measures as being \emph{non-compliant for the rough context}. I seek to replace such non-compliant measures by tuples satisfying the criteria. Based on this heuristic, I would replace the rough inclusion function 
\[k(X, Y)=
\left\{
\begin{array}{ll}
\dfrac{\#(X\cap Y)}{\#(X)},  & \mathrm{if}\,\, X\neq \emptyset \\
1, & \mathrm{else}
\end{array}
\right.\] with 

\[k^{*}(X, Y)=
\left\{
\begin{array}{ll}
\left(y_{1},y_{2}, \ldots, y_{r}\right)   & \mathrm{if}\,\, X^{l}\neq \emptyset, \\
\left(\dfrac{1}{r},\ldots,\dfrac{1}{r}\right), & \mathrm{else}
\end{array}
\right.\]

where \[\dfrac{\#(G_{i})\cdot\chi_{i}(X\cap Y)}{\#(X^{l}) = y_{i}},\;\;i=1, 2, \ldots r.\]

Here it is assumed that $\{G_{1},\ldots , G_{r}\} = \mathcal{G}$ (the collection of granules) and that the function $\chi_{i}$ is being defined via,
\[\chi_{i}(X)=
\left\{
\begin{array}{ll}
1, & \mathrm{if} G_{i}\subseteq X \\
0, & \mathrm{else}
\end{array}\right.\] 

Similarly,
\[k_{1}(X, Y)=
\left\{
\begin{array}{ll}
\dfrac{\#(Y)}{\#(X\cup Y)},  & \mathrm{if}\,\, X\cup Y \neq \emptyset \\
1, & \mathrm{else}
\end{array}\right.\]

can be replaced by 

\[k_{1}^{*}(X, Y)=
\left\{
\begin{array}{ll}
\left( h_1, h_{2}, \ldots,h_r\right)   & \mathrm{if}\, X^{l}\neq \emptyset,\\
\left(\frac{1}{r},\ldots,\frac{1}{r}\right), & \mathrm{else}
\end{array}\right.\]

where \[\dfrac{\#(G_{i})\cdot \chi_{i}(Y)}{\#((X\cup Y)^{l})} = h_{i}\,; i=1,2, ...,r\]

and

\[k_{2}(X, Y)= \dfrac{\#(X^{c}\cup Y)}{\#(S)}\]

can be replaced by 

\[k_{2}^{*}(X, Y)= \left(q_{1}, q_{2},\ldots, q_{r} \right),\]
 
where \[\dfrac{\#(G_{1})\cdot\chi_{1}(X^{c} \cup Y)}{\#(S)} = q_{i},\;\;i=1,2,\ldots,r.\] 

This strategy can be extended to every other non-compliant inclusion function. Addition, multiplication and their partial inverses for natural numbers can be properly generalised to the new types of numbers with special regard to meaning of the operations on elements of dissimilar type. This paves the way for the representation of these general measures in the new number systems (given the granulation).

\section{ON REPRESENTATION OF COUNTS}

In this section, ways of extending the representation of the different types of counts considered in the previous section are touched upon. The basic aim is to endow these with more algebraic structure through higher order constructions. These also correspond to rather superfluous implicit interpretations of counting exact objects by natural numbers. 

Consider the following set of statements:
\begin{description}
\item[A] {Set $S$ has finite cardinality $n$. }
\item[B] {Set $S$ can be counted in $n!$ ways.}
\item[C] {The set of countings $\mathcal{C}(S)$ of the set $S$ has cardinality $n!$.}
\item[E] {The set of countings $\mathcal{C}(S)$ of the set $S$ bijectively corresponds to the permutation group $S_{n}$.}
\item[F] {The set of countings $\mathcal{C}(S)$ can be endowed with a group operation so that it becomes isomorphic to $S_{n}$.}
\end{description}

Statements C, E and F are rarely explicitly stated or used in mathematical practice in the stated form. To prove 'F', it suffices to fix an initial counting. From the point of view of information content it is obvious that $A\subset B\subset C \subset E \subset F$. It is also possible to replace 'set' in the above statements by 'collection' and then from the axiomatic point of statements C, E, and F will require stronger axioms (This aspect will not be taken up in the present paper). I will show that partial algebraic variants of statement F correspond to IPC. 

An important aspect that will not carry over under the generalisation is:
\begin{proposition}
Case F is fully determined by any of the two element generating sets (for finite $n$). 
\end{proposition}

\subsection{Representation Of IPC}

The group $S_{n}$ can be associated with all the usual counting of a collection of $n$ elements. The composition operation can be understood as the action of one counting on another. This group can be associated with the RYS being counted and 'similarly counted pairs' in the latter can be used to generate a partial algebra. Formally, the structure will be deduced using knowledge of $S_{n}$ and then abstracted:  

\begin{itemize}
\item {Form the group $S_{n}$ with operation $\ast$ based on the interpretation of the collection as a finite set of $n$ elements at a higher meta level}
\item {Using the information about similar pairs, associate a second interpretation $s(x)$ based on the IPC procedure with each element $x\in S_{n}$}
\item {Define $(x,y)\in \rho$ if and only if $s(x) = s(y)$}
\end{itemize}

On the quotient $S_{n}|\rho$, let \[a\odot b\,=\,\left\lbrace  \begin{array}{ll}
 c & \mathrm{if}\,\;\{z:\,x\ast y=z,\, x\in a,\,y\in b \}|\rho \in S_{n}|\rho\\
 \mathrm{\infty} & \mathrm{else} \\
 \end{array} \right.\]

The following proposition follows from the form of the definition. 

\begin{proposition}
The partial operation $\odot$ is well defined. 
\end{proposition}

\begin{definition}
A partial algebra of the form $\left\langle S_{n}|\rho,\,\odot\right\rangle $ defined above will be called a \emph{Concrete IPC-partial Algebra} (CIPCA) 
\end{definition}

A partially ordered set $\left\langle F, \leq \right\rangle $ is a \emph{lower semi-linearly ordered set} if and only if:
\begin{itemize}
\item {For each $x$, $x\downarrow$ is linearly ordered.}
\item {$(\forall x, y)(\exists z) z\leq x \wedge z \leq y$ }
\end{itemize}

It is easy to see that IPC counts form lower semi-linearly ordered sets. Structures of these types and properties of their automorphisms are all of interest and will be considered in a separate paper.

\section{FURTHER DIRECTIONS: CONCLUSION}

The broad classes of problems that will be part of future work fall under:
\begin{itemize}
 \item {Improvement of the representation of different classes of countings. The main questions relate to using morphisms or automorphisms in a more streamlined way.}
 \item {More results on identification of granules and definite elements from dialectical counting.}
 \item {Description of other semantics of general \textsf{RST} in terms of counts. This is the basic program of representing all types of general rough semantics in terms of counts.}
 \item {Can the division operation be eliminated in the general measures proposed? In other words, a more natural extensions of fractions (rough rationals) would be of much interest. Obviously this is part of the representation problem for counts.}
 \item {How do algorithms for reduct computation get affected by the generalised measures?}
 \end{itemize}
Apart from these a wide variety of combinatorial questions would be of natural interest.

Originally the presented rough natural number systems had been developed by the present author in conjunction with an axiomatic theory of granules for a much broader programme in the foundations of rough and fuzzy set theories in a separate research paper. The present paper is an adaptation of the number system for generalised measures.     

In this research paper, new methods of counting collections of well defined and indiscernible objects have been advanced with limited use of the axiomatic theory of granules. These new methods of counting have been shown to be applicable to the extension of fundamental measures of \textsf{RST}, rough inclusion measures and consistency degrees of knowledge. The redefined measures possess more information than the original measures and have more realistic orientation with respect to counting. 
\bibliographystyle{IEEEtran}
\bibliography{newsem9996.bib}

%\bibliographystyle{splncs}
%\bibliography{newsem99.bib}

\end{document}